\documentclass[10pt]{amsart}
\usepackage[cp1251]{inputenc}
\usepackage[english,russian]{babel}
\usepackage{amsmath}
\usepackage{amssymb}
\usepackage{amsfonts}

\newtheorem{lemma}{Lemma}
\newtheorem{theorem}{Theorem}

\newtheorem{corollary}{Corollary}
\newtheorem{proposition}{Proposition}

\usepackage{amscd,euscript,amsxtra,amsthm}

\theoremstyle{definition}
\newtheorem{definition}{Definition}
\theoremstyle{convention}

\theoremstyle{example}

\theoremstyle{remark}
\newtheorem{remark}{Remark}

\usepackage[all]{xy}

\numberwithin{equation}{section}

\def\I{{{\mathbb I}}}

\def\E{{{\mathbb E }}}

\def\L{{{\mathbb L }}}

\def\U{{{\mathbb U}}}
\def\V{{{\mathbb V }}}

\def\CC{{{\mathcal C}}}

\def\EE{{{\mathcal E}}}
\def\FF{{{\mathcal F}}}

\def\LL{{{\mathcal L}}}

\def\OO{{{\mathcal O}}}

\def\EExt{{{\mathcal E}xt}}

\def\FFitt{{{\mathcal F}itt}}

\def\Spec{{{\rm Spec \,}}}

\def\Hilb{{{\rm Hilb \,}}}
\def\Univ{{{\rm Univ \,}}}
\def\Grass{{{\rm Grass }}}
\def\Quot{{{\rm Quot \,}}}

\def\Pic{{{\rm Pic \,}}}

\def\coker{{{\rm coker \,}}}
\def\ker{{{\rm ker \,}}}
\def\rank{{{\rm rank \,}}}
\def\im{{{\rm im \,}}}
\def\id{{{\rm id }}}

\begin{document}
\renewcommand{\refname}{References}
\renewcommand{\proofname}{Proof.}
\thispagestyle{empty}

\title[Isomorphism of compactifications of moduli of
vector bundles]{Isomorphism of compactifications of moduli of
vector bundles: nonreduced moduli}
\author{{N.V.Timofeeva}}%
\address{Nadezda Vladimirovna Timofeeva
\newline\hphantom{iii} Yaroslavl State University,
\newline\hphantom{iii} ul. Sovetskaya, 14,
\newline\hphantom{iii} 150000, Yaroslavl, Russia}%
\email{ntimofeeva@list.ru}%
\vspace{1cm} \maketitle {\small
\begin{quote}
\noindent{\sc Abstract. } A morphism of  the moduli functor  of
admissible semistable pairs  to the  Gieseker -- Maruyama moduli
functor (of semistable coherent torsion-free sheaves)  with the
same Hilbert polynomial on the surface, is constructed. It is
shown that these functors are isomorphic, and main components of
moduli scheme for semistable admissible pairs $((\widetilde S,
\widetilde L), \widetilde E)$ are isomorphic to main components of
the Gieseker -- Maruyama moduli scheme.\medskip

\noindent{\bf Keywords:} moduli space, semistable coherent
sheaves, semistable \linebreak admissible pairs, moduli functor,
vector bundles, algebraic surface.
 \end{quote}

\bigskip

\begin{flushright}
{\it To the blessed memory of my Mum}
\end{flushright}

\section*{Introduction}
In the present article we continue to investigate  the
compactification of moduli of stable vector bundles on a surface
by locally free sheaves. Various aspects of its construction and
basic properties were given in  preceding papers of the author
\cite{Tim0} -- \cite{Tim8}.

 In the present article $S$ is smooth irreducible projective
algebraic surface over an algebraically closed field $k$ of
characteristic 0, $\OO_S$ its structure sheaf, $E$ coherent
torsion-free $\OO_S$-module, $E^{\vee}:={{\mathcal
H}om}_{{\mathcal O}_S}(E, {\mathcal O}_S)$ its dual ${\mathcal
O}_S$-module. $E^{\vee}$ is reflexive and hence locally free.
Everywhere in this article a locally free sheaf and its
corresponding vector bundle are idetified and both terms are used
as synonyms. Let  $L$ be very ample invertible sheaf on $S$; it is
fixed and is used as a polarization. The symbol $\chi(\cdot)$
denotes Euler -- Poincar\'{e} characteristic, $c_i(\cdot)$
 $i$-th Chern class.

\begin{definition} \label{admsch} \cite{Tim3, Tim4} Polarized algebraic scheme
$(\widetilde S, \widetilde L)$ is called {\it admissible} if it
satisfies one of the following conditions

i) $(\widetilde S, \widetilde L) \cong (S,L)$,

ii) $\widetilde S \cong {\rm Proj \,} \bigoplus_{s\ge
0}(I[t]+(t))^s/(t^{s+1})$ where $I={{\mathcal F}itt}^0 {{\mathcal
E}xt}^2(\varkappa, {\mathcal O}_S)$ for Artinian quotient sheaf
 $q_0: \bigoplus^r {\mathcal
O}_S\twoheadrightarrow \varkappa$ of length $l(\varkappa)\le c_2$,
and $\widetilde L = L \otimes (\sigma ^{-1} I \cdot {\mathcal
O}_{\widetilde S})$ is very ample invertible sheaf on the scheme
$\widetilde S$; this polarization  $\widetilde L$ is called  {\it
distinguished polarization}.
\end{definition}

Recall the definition of a sheaf of 0-th Fitting ideals known from
commutative algebra and involved in the previous definition. Let
$X$ be a scheme, $F$  $\OO_X$-module of finite presentation $F_1
\stackrel{\varphi}{\longrightarrow} F_0 \to F$. Without loss of
generality we assume that  $\rank F_1 \ge \rank F_0$.
\begin{definition} {\it The sheaf of 0-th Fitting ideals } of
$\OO_X$-module $F$ is defined as  $\FFitt^0 F =\im
(\bigwedge^{\rank F_0} F_1 \otimes \bigwedge^{\rank F_0}
F_0^{\vee} \stackrel{\varphi'}{\longrightarrow}\OO_X)$, where
$\varphi'$ is a morphism of  $\OO_X$-modules induced by $\varphi$.
\end{definition}

\begin{remark} In further considerations we replace $L$ by its
big enough tensor power, if necessary for $\widetilde L$ to be
very ample. This power can be chosen constant and fixed, as shown
in \cite{Tim4}. All Hilbert polynomials are compute according to
new $L$ and $\widetilde L$ respectively.
\end{remark}

As shown in \cite{Tim3}, if $\widetilde S$ satisfies the condition
(ii) in the definition \ref{admsch}, it is decomposed into the
union of several components $\widetilde S=\bigcup_{i\ge
0}\widetilde S_i$. It has a morphism $\sigma: \widetilde S \to S$
which is induced by the structure of $\OO_S$-algebra on the graded
object $\bigoplus_{s\ge 0}(I[t]+(t))^s/(t^{s+1})$.

\begin{definition} \cite{Tim4} $S$-{\it stable}
(respectively, {\it semistable}) {\it pair} $((\widetilde
S,\widetilde L), \widetilde E)$ is the following data:
\begin{itemize}
\item{$\widetilde S=\bigcup_{i\ge 0} \widetilde S_i$ --
admissible scheme, $\sigma: \widetilde S \to S$ morphism which is
called {\it canonical}, $\sigma_i: \widetilde S_i \to S$ its
restrictions on components $\widetilde S_i$, $i\ge 0;$}
\item{$\widetilde E$ vector bundle on the scheme
$\widetilde S$;}
\item{$\widetilde L \in \Pic \widetilde S$ distinguished polarization;}
\end{itemize}
such that
\begin{itemize}
\item{$\chi (\widetilde E \otimes \widetilde
L^n)=rp(n),$ the polynomial $p(n)$ and the rank $r$ of the sheaf
$\widetilde E$ are fixed;}
\item{the sheaf $\widetilde E$ on the scheme $\widetilde S$ is {\it
stable } (respectively, {\it semistable}) {\it due to Gieseker,}
i.e. for any proper subsheaf $\widetilde F \subset \widetilde E$
for $n\gg 0$
\begin{eqnarray*}
\frac{h^0(\widetilde F\otimes \widetilde L^n)}{{\rm rank\,} F}&<&
\frac{h^0(\widetilde E\otimes \widetilde L^n)}{{\rm rank\,} E},
\\ (\mbox{\rm respectively,} \;\;
\frac{h^0(\widetilde F\otimes \widetilde L^n)}{{\rm rank\,}
F}&\leq& \frac{h^0(\widetilde E\otimes \widetilde L^n)}{{\rm
rank\,} E}\;);
\end{eqnarray*}}
\item{on each of additional components $\widetilde S_i, i>0,$
the sheaf $\widetilde E_i:=\widetilde E|_{\widetilde S_i}$ is {\it
quasi-ideal,} i.e. admits a description of the form
\begin{equation}\label{quasiideal}\widetilde E_i=\sigma_i^{\ast}
\ker q_0/tor\!s_i.\end{equation} for some $q_0\in \bigsqcup_{l\le
c_2} {\rm Quot\,}^l \bigoplus^r {\mathcal O}_S$. }\end{itemize}
\end{definition}
The definition of the subsheaf $tor\!s_i$ will be given below.

Pairs $(( \widetilde S, \widetilde L), \widetilde E)$ such that
 $(\widetilde S, \widetilde
L)\cong (S,L)$ will be called {\it $S$-pairs}.

In the series of articles of the author \cite{Tim0}
--- \cite{Tim4} a projective algebraic scheme
$\widetilde M$ is built up as reduced moduli scheme of
$S$-semistable admissible pairs and in \cite{Tim6} it is
constructed as possibly nonreduced moduli space.

The scheme  $\widetilde M$ contains an open subscheme  $\widetilde
M_0$ which is isomorphic to the subscheme $M_0$ of
Gieseker-semistable vector bundles in the Gieseker -- Maruyama
moduli scheme $\overline M$ of torsion-free semistable sheaves
whose Hilbert polynomial is equal to  $\chi(E \otimes L^n)=rp(n)$.
The following definition of Gieseker-semistability is used.

\begin{definition}\cite{Gies} The coherent  ${\mathcal O}_S$-sheaf $E$ is {\it
stable} (respectively, {\it semistable}) if for any proper
subsheaf  $F\subset E$ of rank $r'={\rm rank\,} F$ for $n\gg 0$
$$
\frac{\chi(E \otimes L^n)}{r}>\frac{\chi(F\otimes L^n)}{r'},\;
{\mbox{\LARGE (}}{\mbox{\rm respectively,}} \; \frac{\chi(E
\otimes L^n)}{r}\ge \frac{\chi(F\otimes L^n)}{r'}{\mbox{\LARGE
)}}.
$$
\end{definition}

Let  $E$ be a semistable locally free sheaf. Then, obviously, the
sheaf  $I={{\mathcal F}itt}^0 {{\mathcal E}xt}^1(E, {\mathcal
O}_S)$ is trivial and $\widetilde S \cong S$. In this case
$((\widetilde S, \widetilde L), \widetilde E) \cong ((S, L),E)$
and we have a bijective correspondence $\widetilde M_0 \cong M_0$.

Let $E$ be a semistable nonlocally free coherent sheaf; then the
scheme  $\widetilde S$ contains reduced irreducible component
$\widetilde S_0$ such that the morphism
$\sigma_0:=\sigma|_{\widetilde S_0}: \widetilde S_0 \to S$ is a
morphism of blowing up of the scheme $S$ in the sheaf of ideals
$I= {{\mathcal F}itt}^0 {{\mathcal E}xt}^1(E, {\mathcal O}_S).$
Formation of a sheaf  $I$ is an approach to the characterization
singularities if the sheaf $E$ i.e. its difference from a locally
free sheaf. Indeed, the quotient sheaf $\varkappa:= E^{\vee \vee}/
E$ is Artinian of length not greater then $c_2(E)$, and
${{\mathcal E}xt}^1(E, {\mathcal O}_S) \cong {{\mathcal
E}xt}^2(\varkappa, {\mathcal O}_S).$ Then ${{\mathcal F}itt}^0
{{\mathcal E}xt}^2(\varkappa, {\mathcal O}_S)$ is a sheaf of
ideals of  (in general case nonreduced) subscheme $Z$ of bounded
length \cite{Tim6} supported at finite set of points on the
surface $S$. As it is shown in \cite{Tim3}, others irreducible
components $\widetilde S_i, i>0$ of the scheme  $\widetilde S$ in
general case carry nonreduced scheme structure.

Each semistable coherent torsion-free sheaf $E$ corresponds to a
pair  $((\widetilde S, \widetilde L), \widetilde E)$ where
$(\widetilde S, \widetilde L)$  defined as described.

Now we describe the construction of the subsheaf $tor\!s$ in
(\ref{quasiideal}). Let $U$ be Zariski-open subset in one of
components  $\widetilde S_i, i\ge 0$, and
$\sigma^{\ast}E|_{\widetilde S_i}(U)$ correspond\-ing group of
sections. This group is ${\mathcal O}_{\widetilde S_i}(U)$-module.
Sections  $s\in \sigma^{\ast}E|_{\widetilde S_i}(U)$ annihilated
by prime ideals of positive codimensions in  ${\mathcal
O}_{\widetilde S_i}(U)$, form a submodule in
$\sigma^{\ast}E|_{\widetilde S_i}(U)$. This submodule is denoted
as  $tor\!s_i(U)$. The correspondence
 $U \mapsto tor\!s_i(U)$ defines a subsheaf $tor\!s_i
\subset \sigma^{\ast}E|_{\widetilde S_i}.$ Note that associated
primes of positive codimensions which annihilate sections $s\in
\sigma^{\ast}E|_{\widetilde S_i}(U)$, correspond to subschemes
supported in the preimage $\sigma^{-1}({\rm Supp\,}
\varkappa)=\bigcup_{i>0}\widetilde S_i.$ Since by the construction
the scheme  $\widetilde S=\bigcup_{i\ge 0}\widetilde S_i$ is
connected \cite{Tim3}, subsheaves  $tor\!s_i, i\ge 0,$ allow to
construct a subsheaf $tor\!s \subset \sigma^{\ast}E$. The former
subsheaf is defined as follows. A section  $s\in
\sigma^{\ast}E|_{\widetilde S_i}(U)$ satisfies the condition $s\in
tor\!s|_{\widetilde S_i}(U)$ if and only if
\begin{itemize}
\item{there exist a section  $y\in {\mathcal O}_{\widetilde S_i}(U)$ such that
 $ys=0$,}
\item{at least one of the following two conditions is satisfied:
either  $y\in {\mathfrak p}$, where $\mathfrak p$ is prime ideal
of positive codimension; or there exist Zariski-open subset
$V\subset \widetilde S$ and a section  $s' \in \sigma^{\ast}E (V)$
such that  $V\supset U$, $s'|_U=s$, and $s'|_{V\cap \widetilde
S_0} \in$\linebreak $tor\!s (\sigma^{\ast}E|_{\widetilde
S_0})(V\cap \widetilde S_0)$. In the former expression the torsion
subsheaf $tor\!s(\sigma^{\ast}E|_{\widetilde S_0})$ is understood
in usual sense.}
\end{itemize}

The role of the subsheaf $tor\!s \subset \sigma^{\ast}E$ in our
construction is analogous to the role of torsion subsheaf in the
case of reduced and irreducible base scheme. Since no confusion
occur, the symbol  $tor\!s$ is understood everywhere in described
sense. The subsheaf $tor\!s$ is called a {\it torsion subsheaf}.

In \cite{Tim4} it is proven that sheaves $\sigma^{\ast} E/tor\!s$
are locally free. The sheaf $\widetilde E$ include in the pair
$((\widetilde S, \widetilde L), \widetilde E)$ is defined by the
formula $\widetilde E = \sigma^{\ast}E/tor\!s$. In this
circumstance there is an isomorphism $H^0(\widetilde S, \widetilde
E \otimes \widetilde L) \cong H^0(S,E\otimes L).$

In the same article it was proven that the restriction of the
sheaf  $\widetilde E$ to each of components $\widetilde S_i$,
$i>0,$ is given by the quasi-ideality relation (\ref{quasiideal})
where $q_0: {\mathcal O}_S^{\oplus r}\twoheadrightarrow \varkappa$
is an epimorphism defined by the exact triple  $0\to E \to E^{\vee
\vee} \to \varkappa \to 0$ in view of local freeness of the sheaf
$E^{\vee \vee}$.

Resolution of singularities of a semistable sheaf $E$ can be
globalized in a flat family by means of the construction developed
in various versions in  \cite{Tim1, Tim2, Tim4}. Let $T$ be a
reduced irreducible quasi-projective scheme, ${\mathbb E}$ a sheaf
of  ${\mathcal O}_{T\times S}$-modules, ${\mathbb L}$ invertible
${\mathcal O}_{T\times S}$-sheaf very ample relatively $T$ and
such that  ${\mathbb L}|_{t\times S}=L$, and $\chi({\mathbb E}
\otimes {\mathbb L}^n|_{t\times S})=rp(n)$ for all closed points
$t\in T$. We also assume that $T$ contains nonempty open subset
$T_0$ such that ${\mathbb E}|_{T_0\times S}$ is locally free
${\mathcal O}_{T_0 \times S}$-module. Then following objects are
defined:
\begin{itemize} \item{$\widetilde T$ integral normal scheme
obtained as a blowing up $\phi: \widetilde T \to T$ of the scheme
$T$,}
\item{$\pi: \widetilde \Sigma \to \widetilde T$ flat family of admissible
schemes with invertible ${\mathcal O}_{\widetilde \Sigma}$-module
$\widetilde {\mathbb L}$ such that  $\widetilde {\mathbb
L}|_{t\times S}$ distinguished polarization of the scheme
$\pi^{-1}(t)$,}
\item{$\widetilde {\mathbb E}$ locally free
${\mathcal O}_{\widetilde \Sigma}$-module and $((\pi^{-1}(t),
\widetilde {\mathbb L}|_{\pi^{-1}(t)}),  \widetilde {\mathbb
E}|_{\pi^{-1}(t)})$ is $S$-semistable admiss\-ible pair.}
\end{itemize}
In this situation there is a blowup morphism $\Phi: \widetilde
\Sigma \to
 \widetilde T \times S$ and \begin{equation}\label{descsh}(\Phi_{\ast} \widetilde {\mathbb E})^{\vee
\vee}=(\phi, id_S)^{\ast}{\mathbb E};\end{equation} what follows
from the coincidence of reflexive sheaves at right hand side and
at left hand side, on the open subset apart the subset of
codimension 3. It is important that the scheme  $\widetilde T
\times S$ is integral and normal.

The described mechanism was called in \cite{Tim4} a {\it standard
resolution}.

In  \cite{Tim8} the procedure of standard resolution is
generalized to the case of families with nonreduced base. It is
shown that transformation of the family of torsion-free coherent
sheaves $E$ can be done in such a way that we get a family of
admissible semistable pairs  $((\pi: \widetilde \Sigma \to T,
\widetilde \L), \widetilde \E)$ with the same base $T$, i.e. base
scheme does not undergo a birational transformation and $\phi$ is
identity isomorphism.

\begin{remark} In \cite{Tim8} we did not prove the relation analogous
to (\ref{descsh}). Then when speaking of standard resolution of
the family with nonreduced base we do not mention such a relation.
\end{remark}

In section 1 we remind the definition of the functor  $\mathfrak
f^{GM}$ of moduli of coherent torsion-free sheaves ("Gieseker --
Maruyama functor") (\ref{funcGM}, \ref{famGM}) and improve the
definition of the functor $\mathfrak f$ of moduli of admissible
semistable pairs (\ref{funcmy}, \ref{class}). The rank $r$ and
polynomial $p(n)$ are fixed and equal for both moduli functors.
After that we give the description of the transformation of a
family of semistable admissible pairs
 $((\pi: \widetilde \Sigma \to T, \widetilde \L),
\widetilde \E)$ with (possibly, nonreduced) base scheme $T$ to a
family $\E$ of coherent torsion-free semistable sheaves with the
same base $T$. The transformation provides a morphism of the
functor of admissible semistable pairs ${\mathfrak f}$ to Gieseker
-- Maruyama functor $\mathfrak f^{GM}$.

In section 2 we show that the morphism of functors we constructed
is an inverse for the morphism $\underline \kappa: {\mathfrak
f}^{GM} \to {\mathfrak f}$ built up in  \cite{Tim8}. In this way
the functors of interest  (namely, their subfunctors corresponding
to families containing locally free sheaves and $S$-pairs
respectively) are isomorphic.
\smallskip

In the present article we prove  following results.
\begin{theorem}\label{thfunc} There is a natural transformation
$\underline \tau: {\mathfrak f} \to {\mathfrak f}^{GM}$ of every
maximal closed irreducible subfunctor of the moduli functor of
admissible semistable pairs containing $S$-pairs to the
corresponding maximal closed irreducible subfunctor of Gieseker --
Maruyama moduli functor which contains locally free sheaves with
same rank and Hilbert polynomial. This natural transformation is
inverse to the natural transformation $\underline \kappa$
constructed in \cite{Tim8} and induced by the procedure of
standard resolution developed in the same article.
 Hence both morphisms of nonreduced moduli functors $\underline \kappa:
\mathfrak f^{GM}\to \mathfrak f$ and $\underline \tau: {\mathfrak
f} \to {\mathfrak f}^{GM}$ are isomorphisms.
\end{theorem}
\begin{corollary}\label{thsch}
The union of main components of nonreduced moduli scheme
$\widetilde M$ for ${\mathfrak f}$ is isomorphic to the union of
main components of nonreduced Gieseker -- Maruyama scheme
$\overline M$ for sheaves with same rank and Hilbert polynomial.
\end{corollary}
\smallskip

\section{Morphism of moduli functors}

Following \cite[ch. 2, sect. 2.2]{HL}, we recall some definitions.
Let ${\mathcal C}$ be a category, ${\mathcal C}^o$ its dual,
 ${\mathcal C}'={{\mathcal F}unct}({\mathcal C}^o, Sets)$ category of
 functors to the category of sets. By Yoneda's lemma, the functor  ${\mathcal C} \to
{\mathcal C}': F\mapsto (\underline F: X\mapsto {\rm Hom
\,}_{{\mathcal C}}(X, F))$ includes ${\mathcal C}$ into ${\mathcal
C}'$ as full subcategory.

\begin{definition}\label{corep}\cite[ch. 2, definition 2.2.1]{HL}
The functor  ${\mathfrak f} \in {\OO}b\, \CC'$ is {\it
corepresented by the object} $M \in {\OO}b \,\CC$, if there exist
a $\CC'$-morphism $\psi : {\mathfrak f} \to \underline M$ such
that any morphism $\psi': {\mathfrak f} \to \underline F'$ factors
through the unique morphism  $\omega: \underline M \to \underline
F'$.
\end{definition}

\begin{definition} The scheme  $\widetilde M$ is a {\it coarse moduli space}
for the functor $\mathfrak f$ if  $\mathfrak f$ is corepresented
by the scheme
 $\widetilde M.$
 \end{definition}

Let $T,S$ be  schemes over a field $k$, $\pi: \widetilde \Sigma
\to T$ a morphism of $k$-schemes. We introduce the following

\begin{definition}\label{bitriv} The family of schemes
$\pi: \widetilde \Sigma \to T$ is {\it birationally $S$-trivial}
if there exist isomorphic open subschemes $\widetilde \Sigma_0
\subset \widetilde \Sigma$ and $\Sigma_0 \subset T\times S$ and
there is a scheme equality $\pi(\widetilde \Sigma_0)=T$.
\end{definition}
The former equality means that all fibres of the morphism $\pi$
have nonempty intersections with the open subscheme $\widetilde
\Sigma_0$.

In particular, if $T=\Spec k$ then $\pi$ is a constant morphism
and $\widetilde \Sigma_0 \cong \Sigma_0$ is open subscheme in $S$.

Since in the present paper we consider only $S$-birationally
trivial families, they will be referred to as {\it birationally
trivial} families.

We consider sets of families of semistable pairs
\begin{equation}\label{class}{\mathfrak F}_T= \left\{
\begin{array}{l}\pi: \widetilde \Sigma \to T \mbox{\rm \;\;birationally $S$-trivial
},\\
\widetilde \L\in \Pic \widetilde \Sigma \mbox{\rm \;\;flat over
}T,\\
\mbox{\rm for }m\gg 0 \; \widetilde \L^m \; \mbox{\rm very ample
relatively }T,\\ \forall t\in T \;\widetilde L_t=\widetilde
\L|_{\pi^{-1}(t)}
\mbox{\rm \; ample;}\\
(\pi^{-1}(t),\widetilde L_t) \mbox{\rm \;admissible scheme with
distinguished polarization}; \\
\chi (\widetilde L_t^n) \mbox{\rm \; does not depend on }t,\\
 \widetilde \E \;\; \mbox{\rm locally free } \OO_{\Sigma}-\mbox{\rm
sheaf flat over } T;\\
 \chi(\widetilde \E\otimes\widetilde \L^{n})|_{\pi^{-1}(t)})=
 rp(n);\\
 ((\pi^{-1}(t), \widetilde L_t), \widetilde \E|_{\pi^{-1}(t)}) -
 \mbox{\rm semistable pair}
 \end{array} \right\} \end{equation}  and a functor
\begin{equation}\label{funcmy}\mathfrak f: (Schemes_k)^o
\to (Sets)\end{equation} from the category of $k$-schemes to the
category of sets. It attaches to any scheme $T$ the set of
equivalence classes of families of the form $({\mathfrak
F}_T/\sim).$

The equivalence relation  $\sim$ is defined as follows. Families
$((\pi: \widetilde \Sigma \to T, \widetilde \L), \widetilde \E)$
and \linebreak $((\pi': \widetilde \Sigma \to T, \widetilde \L'),
\widetilde
 \E')$ from the class  $\mathfrak F_T$ are said to be equivalent
 (notation:\linebreak
 $((\pi: \widetilde \Sigma \to T, \widetilde \L),
 \widetilde \E) \sim ((\pi': \widetilde \Sigma \to T, \widetilde \L'), \widetilde
 \E')$) if\\
 1) there exist an isomorphism $ \iota: \widetilde \Sigma \stackrel{\sim}{\longrightarrow}
 \widetilde \Sigma'$ such that the diagram \begin{equation*}
 \xymatrix{\widetilde \Sigma  \ar[rd]_{\pi}\ar[rr]_{\sim}^{\iota}&&\widetilde \Sigma ' \ar[ld]^{\pi'}\\
&T }
 \end{equation*} commutes.\\
 2) There exist line bundles   $L', L''$ on the scheme  $T$ such
 that $\iota^{\ast}\widetilde \E' = \widetilde \E \otimes \pi^{\ast} L',$
 $\iota^{\ast}\widetilde \L' = \widetilde \L \otimes \pi^{\ast} L''.$

Now discuss what is the "size"\, of the maximal under  inclusion
of those open sub\-schemes $\widetilde \Sigma_0$ in a family of
admissible schemes $\widetilde \Sigma$, which are isomorphic to
appropriate open subschemes in $T\times S$ in the definition
\ref{bitriv}. The set $F=\widetilde \Sigma \setminus \widetilde
\Sigma_0$ is closed. If $T_0$ is open subscheme in $T$ whose
points carry fibres isomorphic to $S$, then $\widetilde \Sigma_0
\supsetneqq \pi^{-1}T_0$ (inequality is true because $\pi
(\widetilde \Sigma_0)=T$ in the definition \ref{bitriv}). The
subscheme $\Sigma_0$ which is open in $T\times S$ and isomorphic
to $\widetilde \Sigma_0$, is such that $\Sigma_0 \supsetneqq
T_0\times S$. If $\pi:\widetilde \Sigma \to T$ is family of
admissible schemes then $\widetilde \Sigma_0 \cong \widetilde
\Sigma \setminus F$, and $F$ is (set-theoretically) the union of
additional components of fibres which are non-isomorphic to $S$.


\smallskip

The Gieseker -- Maruyama functor
\begin{equation}\label{funcGM}{\mathfrak f}^{GM}:
(Schemes_k)^o \to Sets,\end{equation} attaches to any scheme $T$
the set of equivalence classes of families of the following form
${\mathfrak F}_T^{GM}/\sim$, where
\begin{equation}\label{famGM} \mathfrak
F_T^{\,GM}= \left\{
\begin{array}{l} \E \;\;\mbox{\rm sheaf of } \OO_{T\times S}-
\mbox{\rm modules flat over } T;\\
\L \;\;\mbox{\rm invertible sheaf of } \OO_{T\times S}-\mbox{\rm
modules,}\\ \mbox{\rm  ample relatively to } T\\
\mbox{\rm and such that } L_t:=\L|_{t\times S}\cong L\; \mbox{\rm for any point } t\in T;\\
E_t:=\E|_{t\times S} \;\mbox{\rm torsion-free and Gieseker-semistable;}\\
\chi(E_t \otimes L_t^n)=rp(n).\end{array}\right\}
\end{equation}

Families   $\E, \L$ and $\E',\L'$ from the class $\mathfrak
F^{GM}_T$ are said to be equivalent (notation: $ (\E, \L)\sim
(\E',\L')$), if there exist linebundles  $L', L''$ on the scheme
$T$ such that $\E' = \E \otimes p^{\ast} L',$ $ \L' =  \L \otimes
p^{\ast} L''$ where $p: T\times S \to T$ is projection onto the
factor.

\begin{remark} Since $\Pic (T \times S) =\Pic T \times \Pic
S$, our definition of the moduli functor ${\mathfrak f}^{GM}$ is
equivalent to the standard definition which can be found, for
example, in \cite{HL}: the difference in choice of polarizations
$\L$ and $\L'$ having isomorphic restrictions on fibres over the
base $T$, is avoided by the equivalence which is induced by
tensoring by inverse image of an invertible sheaf $L''$ from the
base $T$.
\end{remark}

The morphism of functors  $\underline \kappa: \mathfrak f^{GM} \to
\mathfrak f$ в \cite{Tim8} is defined by commutative diagrams
\begin{equation}\label{morfun}\xymatrix{T \ar@{|->}[rd] \ar@{|->}[r]&
\mathfrak F^{GM}_T/\sim \ar[d]\\
& \mathfrak F_T/\sim}
\end{equation}
where $T\in {\mathcal O}b RSch_k$, $\underline
\kappa(T):(\mathfrak F^{GM}_T/\sim) \to ( \mathfrak F_T/\sim)$ is
a morphism in the category of sets (mapping).

\begin{remark}\label{deformability}
We consider the subfunctors in $\mathfrak f^{GM}$ (resp. in
${\mathfrak f}$) which correspond to unions of maximal irreducible
substacks containing locally free sheaves (resp. $S$-pairs). Then
any family $(\L,\E)$ (resp. $((\pi:\widetilde \Sigma \to T,
\widetilde L), \widetilde E)$) with base $T$ can be include into
the family $(\L',\E')$ (resp. $((\pi':\widetilde \Sigma'\to T',
\widetilde \L'), \widetilde \E')$) with some connected and,
possibly, nonreduced base $T'$ and containing locally free sheaves
(resp. $S$-pairs), according to  fibred diagrams
\begin{equation*}\label{deffam}\xymatrix{T \ar@{^(->}[d]_i & \ar[l] T\times S \ar@{^(->}[d]^{(i,id_S)} \\
T' & \ar[l] T' \times S} (\mbox{\rm resp., }
\xymatrix{T \ar@{^(->}[d]_i & \ar[l] \widetilde \Sigma \ar@{^(->}[d]^{\widetilde i} \\
T' & \ar[l] \widetilde  \Sigma'})
\end{equation*}
Namely, $\E=(i, id_S)^{\ast}\E'$ (resp., $\widetilde \Sigma=
\widetilde \Sigma' \times _{T'} T$, $\widetilde i: \widetilde
\Sigma \hookrightarrow \widetilde \Sigma'$ is induced morphism of
immersion, $\widetilde \E = \widetilde i^{\ast} \widetilde \E',$
$\widetilde \L= \widetilde i^{\ast} \widetilde \L'$). We can
assume that  $T'$ is such that its reduction $T'_{red}$ is an
irreducible scheme. This means that we consider admissible
semistable pairs which are deformation equivalent to $S$-pairs
\cite{Tim4}.

We mean under the Gieseker  -- Maruyama scheme $\overline M$ the
union of those components of nonreduced moduli scheme for
semistable coherent torsion-free sheaves, which contain locally
free sheaves, and under the moduli scheme $\widetilde M$ the union
of its components containing $S$-pairs.
\end{remark}

Further we show that there is a morphism of the nonreduced moduli
functor of admiss\-ible semistable pairs to the nonreduced
Gieseker -- Maruyama moduli functor. Namely, for any scheme $T$ we
build up a correspondence $((\pi:\widetilde \Sigma \to T,
\widetilde \L), \widetilde \E)\mapsto (\L, \E)$. It defines a set
mapping
 $(\{((\pi:\widetilde \Sigma \to T, \widetilde
\L), \widetilde \E)\}/\sim)\to (\{\L,\E\}/\sim)$. This means that
the family of semistable coherent torsion-free sheaves $\E$ with
the same base $T$ can be constructed by any family
$((\pi:\widetilde \Sigma \to T, \widetilde \L), \widetilde \E)$ of
admissible semistable pairs which is birationally trivial and flat
over $T$.

Let  $((\widetilde \Sigma, \widetilde \L), \widetilde \E)$ be
birationally trivial family of admissible semistable pairs with
base scheme  $T$. We assume that the scheme $T_{red}$ is
irreducible and contains at least one closed point corresponding
to $S$-pair, i.e. the point  $x \in T$ such that
$\pi^{-1}(x)=\widetilde S_x\cong S.$ The former condition is
provided by the remark \ref{deformability}. In such a point
$\widetilde E|_{\pi^{-1}(x)}=\widetilde E_x$ is locally free sheaf
on the surface  $S$. It is Gieseker-semistable with respect to the
polarization $\widetilde \L|_{\widetilde S_x}=\widetilde L_x\cong
L$.  Let  $\widetilde \Sigma_0$ be the maximal open subscheme in
$\widetilde \Sigma$ which is isomorphic to an open subscheme of
the product $T\times S$. Choose  $m\gg 0$ such that the morphism
of $\OO_{\widetilde \Sigma}$-modules $\pi^{\ast} \pi_{\ast}
(\widetilde \E \otimes \widetilde \L^m) \to \widetilde \E \otimes
\widetilde \L^m$ is surjective. After tensoring $\widetilde \E$ by
an appropriate invertible sheaf from the base $T$ if necessary,
for locally free $\OO_T$-sheaf $\V :=\pi_{\ast}( \widetilde \E
\otimes \widetilde \L^m)$ we have $\pi^{\ast} \V \otimes
\widetilde \L^{-m} |_{\widetilde \Sigma_0} \cong \V \boxtimes
L^{-m}|_{\Sigma_0}$ in view of the isomorphism $\widetilde
\Sigma_0 \cong \Sigma_0$, and there is an epimorphism of
$\OO_{\widetilde \Sigma_0}$-modules $\V \boxtimes L^{-m}|_{
\Sigma_0}\twoheadrightarrow \widetilde \E|_{\widetilde \Sigma_0}.$
Recall that if $T_0$ is a subset of those points in $T$ which
correspond to $S$-pairs then $\pi^{-1}T_0 \subsetneqq \widetilde
\Sigma_0$.

Consider relative Grothendieck' scheme  $\Quot^{rp(n)} (\V
\boxtimes L^{-m}) \to T$. It carries a universal
$\OO_{\Quot^{rp(n)} (\V \boxtimes L^{-m}) \times S}$-quotient
sheaf
$$\V \boxtimes_T \OO_{\Quot^{rp(n)} (\V \boxtimes L^{-m})}\boxtimes
L^{-m} \twoheadrightarrow \E_{\Quot}.$$

The morphism  $\pi^{\ast} \V \twoheadrightarrow \widetilde \E
\otimes \widetilde \L^m$ induces a morphism of  $T$-schemes
$$\widetilde \Sigma_0 \to \Quot^{rp(n)} (\V \boxtimes L^{-m})
\times S,$$ which is locally closed immersion.

Let $T' \subset \Quot^{rp(n)} (\V \boxtimes L^{-m})$ be (possibly
nonreduced) subscheme formed by all quotient sheaves of the form
$q_t: V \otimes L^{-m}\twoheadrightarrow E_t$ such that
 $q_t|_{\Sigma_0 \cap (t\times S)}$ is isomorphic to $(\V
\boxtimes L^{-m}) |_{\Sigma_0\cap (t\times S)}\twoheadrightarrow
\widetilde \E|_{\Sigma_0\cap (t\times S)}.$ The symbol $V$ denotes
a  $k$-vector space  $V\cong H^0(\widetilde S_t=\pi^{-1}(t),
\widetilde E_t \otimes \widetilde L_t^m)$ of dimension $rp(m)$
which is isomorphic to the fibre of the vector bundle  $\V$ at a
point $t\in T$. Equivalently, $T'$ is a scheme-theoretic image of
the subscheme  $\widetilde \Sigma_0$ in $\Quot^{rp(n)} (\V
\boxtimes L^{-m})$. We have the following commutative diagram of
$T$-schemes with fibred square
\begin{equation*}\xymatrix{
\widetilde \Sigma_0\ar[dr]\ar@{^(->}[r]& T'\times S \ar[d]
\ar@{^(->}[r]&\Quot^{rp(n)} (\V \boxtimes L^{-m}) \times S
\ar[d]\\
&T'\ar@{^(->}[r] \ar@{->>}[rd]_{\tau}& \Quot^{rp(n)} (\V \boxtimes
L^{-m}) \ar[d]\\
&&T}
\end{equation*}
The double arrow in the diagram of schemes  means that the
scheme-theoretic image of the morphism $\tau$ coincides with its
target, i.e. the image of the scheme $T'$ under the morphism
$\tau$ has same scheme structure as $T$. It it true because the
image of the subscheme $\widetilde \Sigma_0$ under the projection
on  $T$ coincides with $T$.

We claim that  $\tau$ is an isomorphism. For the proof consider at
first a closed point ${\mathfrak m} \in T$, the image of the fibre
$\widetilde \Sigma_0\cap \pi^{-1}(\mathfrak m)$ in $\Quot^{rp(n)}
(\V \boxtimes L^{-m})\times S$ (which is denoted also as
$\widetilde \Sigma_0\cap \pi^{-1}(\mathfrak m)$) and an
epimorphism corresponding to the point  $\mathfrak m$
\begin{equation}\label{forcont}\V \boxtimes L^{-m}
|_{\widetilde \Sigma_0\cap \pi^{-1}(\mathfrak
m)}\twoheadrightarrow \E|_{\widetilde \Sigma_0\cap
\pi^{-1}(\mathfrak m)}.\end{equation} Note that in the situation
of the remark \ref{deformability} the image of the subset
$\pi^{-1}(\mathfrak m) \setminus (\widetilde \Sigma_0\cap
\pi^{-1}(\mathfrak m))$ in  $S$ is a finite collection of points
on the surface $S={\mathfrak m} \times S\subset \Quot^{rp(n)} (\V
\boxtimes L^{-m}) \times S$. Indeed, this is a "limit"\, of closed
immersions of surfaces isomorphic to $S$ with semistable locally
free sheaves on them (in the sense of Gieseker -- Maruyama
functor). It is a surface isomorphic to $S$ again, with semistable
coherent torsion-free sheaf on it.

The subscheme $U:=(\widetilde \Sigma_0\cap \pi^{-1}(\mathfrak m))
\subset {\mathfrak m} \times S$ is non-affine and strictly greater
then any proper affine subscheme in $S$ which does not contain the
subset $S \setminus U$.

We will show that the morphism (\ref{forcont}) has a unique
continuation to the whole of the subscheme  ${\mathfrak m}\times
S$. This means that the submodule $$\ker (\V \boxtimes
L^{-m}|_{\widetilde \Sigma_0\cap \pi^{-1}(\mathfrak m)} \to
\E_{\Quot}|_{\widetilde \Sigma_0\cap \pi^{-1}(\mathfrak m)})$$ has
a unique continuation to the whole of the subscheme  ${\mathfrak
m}\times S.$ In the sequel we use the notation
$E:=\E_{\Quot}|_{{\mathfrak m}\times S}.$

For any open  $U' \subset S$ such that  $U' \cap (S \setminus U)
\ne \emptyset$, the element  $f\in (V\otimes
L^{-m})(U')=\V\boxtimes L^{-m}|_{\mathfrak m \times S}(U')$
vanishing in $E(U'\cap U),$ maps to $0\in E(U')$ on the whole of
$U'$. Then in the fibre over the closed point ${\mathfrak m} \in
T$ there is a unique continuation of the epimorphism  $V \otimes
L^{-m}|_U \twoheadrightarrow E|_U$ to the homomorphism  $V \otimes
L^{-m} \twoheadrightarrow E.$ Since the morphism $\kappa_{red}:
\overline M \to \widetilde M$ is bijective \cite{Tim4} and any
semistable pair $((\widetilde S, \widetilde L), \widetilde E)$
corresponds to a coherent semistable torsion-free sheaf $E$, and
for  $m\gg 0$ the homomorphism we built $V \otimes L^{-m}
\twoheadrightarrow E$ defines a point in $Q \subset \Quot^{rp(n)}
(V \otimes L^{-m})$, then this continuing homomorphism is an
epimorphism.

Now turn to the continuation of the epimorphism $\V \boxtimes
L^{-m}|_{\widetilde \Sigma_0}\twoheadrightarrow
\E_{\Quot}|_{\widetilde \Sigma_0}$. Let $\U' \subset T'\times S$
be an open subscheme such that $\U' \cap \widetilde \Sigma_0\ne
\emptyset.$ Assume that there exist an element $f\in (\V \boxtimes
L^{-m})(\U')$ vanishing in $\E_{\Quot}|_{T'\times S}(\U' \cap
\widetilde \Sigma_0)$ but not in  $\E_{\Quot}|_{T'\times S}(\U').$
This means that $f|_{\U'\cap (T'\times S \setminus \widetilde
\Sigma_0)}\ne 0$ what leads to the decomposition of irreducible
topological space $T'\times S$ into the disjoint union of two open
subsets. This contradiction proves unique continuation of the
epimorphism
 $\V \boxtimes_T \OO_{\Quot^{rp(n)} (\V \boxtimes L^{-m})}\boxtimes L^{-m}|_{\widetilde
\Sigma_0}\twoheadrightarrow \E_{\Quot}|_{\widetilde \Sigma_0}$ to
the epimorphism $\V \boxtimes_T\OO_{\Quot^{rp(n)} (\V \boxtimes
L^{-m})}\boxtimes L^{-m}|_{T'\times S}\twoheadrightarrow \E$ where
$\E:=\E_{\Quot}|_{T'\times S}$.

Note that the correspondence we built $T \mapsto T'$ is functorial
and yields in a morphism of functors  $Mor_{Sch_k} (-, T) \to
Mor_{Sch_k} (-, T')$ and hence in natural transformation of
functors of points for schemes $T$ and $T'$. This means
\cite[lecture 3, Proposition]{Mum} that there is a morphism of
schemes  $\tau^{-1}: T \to T'$ which is inverse to $\tau$.

It rests to confirm ourselves that the subscheme  $T'\subset
\Quot^{rp(n)} (\V \boxtimes L^{-m})$ in whole lies in the
subscheme $Q$ corresponding to semistable coherent torsion-free
sheaves. For this purpose we assume that  $T'=\Spec A$ for $A$
being a local $k$-algebra of finite type  with a maximal ideal
${\mathfrak m}$. The closed point  ${\mathfrak m} \in T'$
corresponding to admissible semistable pair $((\widetilde S,
\widetilde L), \widetilde E)$, is taken in our construction to a
coherent $L$-semistable torsion-free sheaf  $E_{\mathfrak m}$.
Then, passing to  localizations of the ring  $A$ in each of its
prime ideals ${\mathfrak p}\in \Spec A=T'$, we can conclude that
all semistable admissible pairs which correspond to points of the
scheme $T'$, are taken to semistable torsion-free coherent
sheaves. Then $T'_{red}$ belongs to the subscheme $Q$ of
semistable coherent torsion-free sheaves. Properties of
torsion-freeness and of Gieseker-semistability are open in flat
families of coherent sheaves. Then if  $T'_{red}$ belongs to the
subscheme of semistable torsion-free sheaves the same is true for
$T'$.

Indeed, assume that  $T'$ does not belong to the subscheme of
torsion-free semistable sheaves. Since $E_{\mathfrak m}$ is
semistable and torsion-free, then there exists a nonempty closed
subscheme in  $T'$ such that it contains a closed point
corresponding to non-semistable sheaf or to a sheaf with torsion.
This is impossible because the only closed point of the scheme
$T'$ corresponds to semistable torsion-free sheaf.

We have proven that there is a natural transformation $\underline
\tau: {\mathfrak f} \to {\mathfrak f}^{GM}$ of the functor of
admissible semistable pairs to the functor of semistable coherent
torsion-free sheaves. This natural transformation is defined by
the series if commutative diagrams
\begin{equation*}\xymatrix{T \ar@{=}[d]\ar
@{|->}[r]^{\mathfrak f}&{\mathfrak F}_T/\sim \ar[d]\\
 T \ar@{|->}[r]^{\mathfrak f^{GM}}&{\mathfrak
F}^{GM}_{T}/\sim}
\end{equation*}
and leads to the morphism of moduli schemes  $\tau: \widetilde M
\to \overline M$ by well-known procedure. In particular, the
deduction of a morphism of moduli schemes from the morphism of
functors can be found in \cite{Tim8}.
\begin{remark} In \cite{Tim8} the natural transformation
$\underline \kappa: {\mathfrak f}^{GM}\to {\mathfrak f}$ and the
corresponding morphism of moduli schemes  $\kappa: \overline M \to
\widetilde M$ are built up.
\end{remark}

\section{Isomorphism of moduli functors}

The construction of the previous section establishes the morphism
of functors \linebreak $\underline \tau:{\mathfrak f } \to
{\mathfrak f}^{GM}$. In the paper  \cite{Tim8} the morphism in
opposite direction $\underline \kappa$ is constructed. It is
necessary to prove that these two morphisms are mutually inverse
and hence provide an isomorphism of functors.

First we show that $\underline \tau \circ {\underline
\kappa}=id_{{\mathfrak f}^{GM}}$. For this purpose take a family
of semistable torsion-free coherent sheaves $\E$ and a family of
polarizations $\L$. Tensoring these sheaves if necessary by
appropriate invertible $\OO_T$-sheaves, assume that locally free
sheaves \linebreak  $p_{\ast}(\E \otimes \L^m)$ and  $p_{\ast}
\L^m$ have 1st Chern class equal to 0. Apply the procedure of
standard resolution from \cite{Tim8} to the family chosen. This
leads to a family of admissible semistable pairs $((\pi:
\widetilde \Sigma \to T, \widetilde \L), \widetilde \E).$ Now
perform a transformation from sect.1 of the present paper. We get
a family of coherent torsion-free semistable sheaves  $\E''$ and a
family of polarizations  $\L''$ again. Now, tensoring both
$\OO_{T\times S}$-modules by appropriate invertible sheaves form
the base $T$ and getting $\OO_{T\times S}$-modules $\E'$ and $\L'$
respectively, we consider locally free $\OO_T$-sheaves
$p_{\ast}(\E' \otimes {\L'}^m)$ and $p_{\ast}{\L'}^m$ as having
1st Chern class equal to 0. Families  $(\L,\E)$ and $(\L', \E')$
coincide along the subscheme  $\Sigma_0$ in the notation of
sect.1. According to the reasoning in sect. 1, they coincide on
the whole of the product $\Sigma=T\times S$. From this we conclude
that $(\L,\E) \sim (\L', \E').$ The proof done means that that the
natural transformation $\underline \kappa: {\mathfrak f}^{GM} \to
{\mathfrak f}$ is a section of the natural transformation
$\underline \tau$ and hence the morphism of moduli spaces $\kappa:
\overline M \to \widetilde M$ is a section of the morphism
$\tau:\widetilde M \to \overline M.$ Then $\kappa: \overline M \to
\widetilde M$ is a closed immersion. This follows from the simple
lemma.
\begin{lemma} Let  $f: X\to Y$ be a scheme morphism and  $s: Y \to X$ its section.
Then $s$ is closed immersion.
\end{lemma}
\begin{proof} Since $f\circ s=\id_Y$, the morphism $s$ maps the scheme
$Y$ isomorphically to its image in $X$. Pass to affine subschemes
and assume that $X=\Spec A$, $Y=\Spec B$, where $A,B$ are
commutative rings, $f^{\sharp}: B \to A,$ $s^{\sharp}: A \to B$
their homomorphisms inducing corresponding scheme morphisms. Now
$s^{\sharp} \circ f^{\sharp}=\id_B$. Then $f^{\sharp}$ maps $B$
isomorphically to its image in  $A$ and  $s^{\sharp}$ maps a
subring  $f^{\sharp}(B) \subset A$ to $B$. Hence $s^{\sharp}$ is
surjective homomorphism and  $B$ is isomorphic to a quotient ring
of  $A$. Now conclude that  $s: \Spec B \to \Spec A$ is closed
immersion.
\end{proof}

We have proven that any admissible semistable pair $((\widetilde
S, \widetilde L), \widetilde E)$ corresponds to a coherent
semistable torsion-free sheaf $E$, and there is a closed immersion
$\overline M \hookrightarrow \widetilde M.$

Now confirm that  ${\underline \kappa} \circ \underline \tau
=id_{\mathfrak f}$. Let  $T_0$ be maximal nonempty open subscheme
in  $T$ whose closed points correspond to $S$-pairs. Restriction
of the family $((\widetilde \Sigma \to T, \widetilde \L),
\widetilde E)$ to $T_0$ induces a locally closed immersion $\mu_0:
T_0 \hookrightarrow \Quot^{rp(n)} (\V \boxtimes L^{-m})$. This
immersion in the composite with the structure projection to the
base  $T$ provides an isomorphism $\mu_0(T_0) \cong T_0$. In this
circumstance  $\mu_0^{\ast} \E_{\Quot}=\widetilde
\E|_{\pi^{-1}(T_0)}.$

Form a Grassmannian bundle  $\Grass (\V, r) \to T$ of $r$-quotient
spaces of fibres of the vector bundle $\V$. The fibre of the
bundle $\Grass (\V, r) \to T$ at the point $t\in T$ is usual
Grassman variety  $G(V_t, r)$. Since all vector spaces  $V_t\cong
V$ are isomorphic as having equal dimensions, all fibres of the
Grassmannian bundle are also isomorphic: $G(V_t,r) \cong G(V,r)$.
For an admissible semistable pair  $((\widetilde S, \widetilde L),
\widetilde E)$ for  $m\gg 0$ there is a closed immersion
\linebreak
 $j: \widetilde S \hookrightarrow
G(V,r)$ which is defined by the epimorphism of locally free
sheaves \linebreak $H^0(\widetilde S, \widetilde E \otimes
\widetilde L^m) \boxtimes \widetilde L^{-m}\twoheadrightarrow
\widetilde E$. Let $\OO_{G(V,r)}(1)$ be a positive generator of
the group $\Pic G(V,r)$, then $P(n):= \chi (j^{\ast}
\OO_{G(V,r)}(n))$ is the Hilbert polynomial of the closed
subscheme $j(\widetilde S) $. Fix the polynomial  $P(n)$. Consider
the Hilbert scheme $\Hilb^{P(n)}\Grass(\V,r)$ of subschemes in
$\Grass (\V,r)$ having Hilbert polynomial equal to $P(n)$, and its
universal subscheme $\Univ^{P(n)}\Grass(\V,r) \to
\Hilb^{P(n)}\Grass(\V,r)$.

The family of admissible semistable pairs  $((\pi:\widetilde
\Sigma \to T, \widetilde \L), \widetilde \E)$ induces
the following diagram with fibred
square and immersions with "relative"\, schemes
\begin{equation*}\xymatrix{ \widetilde \Sigma \ar[d]_{\pi}
\ar@{^(->}[r]^<<<<<<{\widetilde \mu}&
\Univ^{P(n)}\Grass(\V,r) \ar[d]\\
T \ar@{^(->}[r]^<<<<<<{\mu} \ar[rd]^{\sim}&
\Hilb^{P(n)}\Grass(\V,r) \ar[d]\\
&T}
\end{equation*}
The family of interest contains $S$-pairs; take a maximal open
 $\widetilde \Sigma_0 \subset
\widetilde \Sigma$ which is isomorphic to an open subset in the
product $T \times S$. Consider the relative Grothendieck' scheme
$\Quot^{rp(n)}(\V \boxtimes L^{-m})$ and the diagonal immersion
$$\delta_0: \widetilde \Sigma_0 \hookrightarrow
\Univ^{P(n)}\Grass(\V,r) \times _T \Quot^{rp(n)}(\V \boxtimes
L^{-m}).$$ Let a morphism $\varphi$ be defined as a composite of
 $T$-morphisms
\begin{equation*}\xymatrix{\widetilde \Sigma_0 \ar@{^(->}[r]^<<<<<<{\delta_0}
\ar[rd]_\varphi &\Univ^{P(n)} \Grass(\V,r) \!\times _T\!
\Quot^{rp(n)}(\V \boxtimes L^{-m})\ar[d]^{pr_2}\\
&\Quot^{rp(n)}(\V \boxtimes L^{-m} )}
\end{equation*}
Due to the reasoning on the natural transformation $\mathfrak t$
in the previous section, the immersion $\delta_0$ is continued to
the immersion
$$\delta: \widetilde \Sigma \hookrightarrow
\Univ^{P(n)}\Grass(\V,r) \times _T \Quot^{rp(n)}(\V \boxtimes
L^{-m}).$$ The morphism  $\varphi$ has the scheme $T'$ as its
image. It is isomorphic to $T$. The isomorphism is provided by the
projection on the base $T$ of relative schemes.

The immersion $\mu':\varphi(\widetilde \Sigma_0)\cong
T\hookrightarrow \Quot^{rp(n)}(\V \boxtimes L^{-m})$ gives rise to
a family of coherent semistable torsion-free sheaves
$\E={\mu'}^{\ast}\E_{\Quot}$. The family of polarizations
 $\L$ is given, for example, by the formula  $\L=\OO_T
\boxtimes L$.

Now perform a standard resolution of the family $\E$ as described
in \cite{Tim8}. We are interested in such a version of standard
resolution which does not change the base $T$. The standard
resolution suggested in \cite{Tim8} involves a blowing up  $\sigma
\!\!\! \sigma: \widehat \Sigma \to \Sigma$ of the product
$\Sigma=T\times S$ in the sheaf of Fitting ideals $\I=\FFitt^0
\EExt^1(\E, \OO_{\Sigma}).$ In the same article we proved that the
composite $f:=p_1\circ \sigma \!\!\! \sigma: \widehat \Sigma \to
T$ is flat morphism. We obtain a family of admissible semistable
pairs $((\pi': \widetilde \Sigma' \to T, \widetilde \L'),
\widetilde \E')$, where $\widetilde \Sigma'=\widehat \Sigma,$
$\pi'=f$, $\widetilde \L'= \sigma\!\!\!\sigma^{\ast} \L \otimes
\sigma\!\!\!\sigma^{-1} \I \cdot \OO_{\widetilde \Sigma},$
$\widetilde \E'$ is such as described in \cite{Tim8}. It gives
rise to a locally free sheaf $\pi'_{\ast}(\widetilde \E'\otimes
(\widetilde \L')^m)$. We need the following proposition which will
be proven below.
\begin{proposition} \label{3bundles} After tensoring by appropriate
invertible $\OO_T$-sheaves following locally free sheaves are
isomorphic: $\V=\pi_{\ast}(\widetilde \E \otimes \widetilde
\L^m)$, $\V_0=p_{\ast}(\E \otimes \L^m)$, and
$\V'=\pi'_{\ast}(\widetilde \E' \otimes (\widetilde \L')^m).$
\end{proposition} By the proposition \ref{3bundles} there are
following immersions of $T$-schemes
$$\delta':\widetilde
\Sigma'\hookrightarrow \Univ^{P(n)} \Grass (\V, r) \times _T
\Quot^{rp(n)}(\V \boxtimes L^{-m})$$ and $$\widetilde
\mu':\widetilde \Sigma'\hookrightarrow \Univ^{P(n)} \Grass (\V,
r).
$$ The scheme  $\widetilde \Sigma'$ contains an open subscheme
$\widetilde \Sigma'_0$ which is isomorphic to the subscheme
$\widetilde \Sigma_0$. By this isomorphism and by the construction
of standard resolution there are isomorphisms $\widetilde
\L'|_{\widetilde \Sigma'_0} \cong \widetilde \L|_{\widetilde
\Sigma_0}$ and $\widetilde \E'|_{\widetilde \Sigma'_0} \cong
\widetilde \E|_{\widetilde \Sigma_0}$. Besides, fibres of families
 $\widetilde \Sigma$ and $\widetilde
\Sigma'$ in closed points coincide.

Then in  $\Univ^{P(n)} \Grass (\V, r)$ there are two
$T$-subschemes $\widetilde \mu(\widetilde \Sigma)$ and $\widetilde
\mu'(\widetilde \Sigma)$ which coincide when restricted on the
reduction  $T_{red}$ and on open subsets  $\widetilde
\mu(\widetilde \Sigma_0) = \widetilde \mu'(\widetilde \Sigma'_0)$.

\begin{remark} The coincidence along reductions follows from
the uniqueness of the scheme closure for  $\widetilde
\mu(\widetilde \Sigma_{0red})=\widetilde \mu'(\widetilde
\Sigma'_{0red})$ in $\Univ^{P(n)}\Grass (\V, r) \times _T
T_{red}$. This implies the iso\-morphism of reduced moduli
functors ${\mathfrak f}^{GM}_{red}\cong {\mathfrak f}_{red}$ and
hence the isomorphism of reduced moduli schemes  $\overline
M_{red} \cong \widetilde M_{red}$.
\end{remark}

Now note that under the projection
$$\pi: \Univ^{P(n)}\Grass(\V,r) \to \Hilb^{P(n)}\Grass(\V,r)$$
we have $$ \pi(\widetilde \mu(\widetilde
\Sigma_0))=\mu(T)=\mu'(T)=\pi(\widetilde \mu'(\widetilde
\Sigma'_0)).
$$
This subscheme is mapped  by the structure projection
$\Hilb^{P(n)}\Grass (\V,r) \to T$ to the base $T$ isomorphically.

By the construction and by the universal property of the Hilbert
scheme we have
$$\widetilde \mu(\widetilde \Sigma)=\pi^{-1}\pi(\widetilde \mu(\widetilde \Sigma_0))
=\pi^{-1}\mu(T)=\pi^{-1}\mu'(T)=\pi^{-1}\pi(\widetilde
\mu'(\widetilde \Sigma'_0))=\widetilde \mu'(\widetilde \Sigma').
$$
Isomorphisms $\widetilde \Sigma \cong \widetilde \mu(\widetilde
\Sigma)$ and $\widetilde \Sigma' \cong \widetilde \mu'(\widetilde
\Sigma')$ complete the proof.

\begin{proof}[Proof of proposition \ref{3bundles}]
Consider an epimorphism of  $\OO_{T\times S}$-modules
$$
\V \boxtimes L^{-m}\twoheadrightarrow\E,
$$
associated with the immersion $T \hookrightarrow \Quot^{rp(n)}(\V
\boxtimes L^{-m}).$ Tensoring by $\OO_S$-sheaf $L^m$ and formation
of a direct image  $p_{\ast}$ lead to the morphism of locally free
 $\OO_T$-modules $\psi: \V
\to p_{\ast} (\E \boxtimes L^m)$. The sheaf on the right hand side
differs from  $\V_0$ by tensoring by some invertible
$\OO_T$-module what proves the proposition for  $\OO_T$-modules
$\V$ and $\V_0$.

Now turn to the pair  $\V_0$ and $\V'$. Recall that sheaves
$\widetilde \E'$ and $\widetilde \L'$ are obtained by standard
resolution of the family  $\E$. In this procedure one gets an
epimorphism $\sigma\!\!\!\sigma^{\ast}\E \twoheadrightarrow
\widetilde \E'$ \cite{Tim8}. Twisting by  $(\widetilde \L')^m$ and
formation of the direct image $\sigma\!\!\!\sigma_{\ast}$ lead to
the morphism of $\OO_{T\times S}$-modules
\begin{equation}\label{quasisig}\sigma\!\!\!\sigma_{\ast}(\sigma\!\!\!\sigma^{\ast}\E \otimes
(\widetilde \L')^m) \to \sigma\!\!\!\sigma_{\ast}(\widetilde \E'
\otimes (\widetilde \L')^m).\end{equation} Now we need the
following lemma which generalizes well-known projection formula.
\begin{lemma}\label{profor} Let $f: (X,\OO_X) \to (Y, \OO_Y)$ be a morphism of
locally ringed spaces such that $f_{\ast}\OO_X=\OO_Y$, $\EE$
$\OO_Y$-module of finite presentation, $\FF$  $\OO_X$-module. Then
there is a monomorphism $\EE \otimes f_{\ast}\FF \hookrightarrow
f_{\ast}[f^{\ast}\EE \otimes \FF].$
\end{lemma}
\begin{proof}[Proof of lemma \ref{profor}]
Fix any finite presentation for $\EE$: $ E_1 \to E_0 \to \EE \to
0, $ where $E_0, E_1$ are locally free $\OO_Y$-modules. Formation
of an inverse image  $f^{\ast}$, tensoring by $\otimes_X \FF$
followed by formation of a direct image $f_{\ast}$ lead to a
complex
$$
\dots \to f_{\ast}[f^{\ast}E_1 \otimes_X \FF] \to
f_{\ast}[f^{\ast}E_1 \otimes_X \FF] \to f_{\ast}[f^{\ast}\EE
\otimes_X \FF] \to \dots
$$
Due to the usual projection formula, first two terms equal $E_1
\otimes_Y f_{\ast}\FF$ and $E_0 \otimes_Y f_{\ast}\FF$
respect\-ively. Then we have  $$\EE \otimes_Y f_{\ast}\FF =\coker
(E_1 \otimes_Y f_{\ast}\FF \to E_0 \otimes_Y f_{\ast}\FF)
\hookrightarrow f_{\ast}[f^{\ast}\EE \otimes_X \FF].$$ This proves
the lemma.
\end{proof}

Applying the lemma we get the monomorphism
\begin{equation}\label{projfor}(\E\otimes \L^m) \otimes
\sigma\!\!\!\sigma_{\ast}(\sigma\!\!\!\sigma^{-1}\I \cdot
\OO_{\widetilde \Sigma})^m \hookrightarrow
\sigma\!\!\!\sigma_{\ast}(\sigma\!\!\!\sigma^{\ast}\E \otimes
(\widetilde \L')^m).\end{equation} Formation of inverse image
$p_{\ast}$ in both morphisms  (\ref{quasisig}) and (\ref{projfor})
and the equality  $\pi=p\circ \sigma\!\!\!\sigma$ lead to the
diagram
\begin{equation}\label{di}
\xymatrix{p_{\ast}[(\E\otimes \L^m) \otimes
\sigma\!\!\!\sigma_{\ast}(\sigma\!\!\!\sigma^{-1}\I \cdot
\OO_{\widetilde \Sigma})^m] \ar@{^(->}[d] \ar[rd]^{\eta} \ar[r]&
p_{\ast}[\E\otimes \L^{m}]\\
\pi_{\ast}[\sigma\!\!\!\sigma^{\ast}\E \otimes (\widetilde \L')^m]
\ar[r]& \pi_{\ast}[\widetilde \E' \otimes (\widetilde \L')^m]}
\end{equation}
Upper horizontal arrow is induced by the inclusion
$\sigma\!\!\!\sigma^{-1} \I \cdot \OO_{\widetilde
\Sigma}\hookrightarrow \OO_{\widetilde \Sigma}.$ Lower horizontal
arrow is an epimorphism since  $m\gg 0$ and $\widetilde \L'$ is
ample relatively to the projection $\pi.$ Skew arrow is defined as
a composite of morphisms and will be of use below.

In the diagram (\ref{di}) sheaves from right hand side are locally
free of rank $rp(m)$. Tensoring $\E$ (or $\widetilde \E'$) by an
appropriate invertible sheaf  $\LL$ from the base $T$ we make
sheaves from the right hand side in (\ref{di}) coincide on the
open subset out of a subscheme of codimension $\ge 2$. Besides, it
is known \cite{Hart} that for a sheaf of ideals  $\I$ on a scheme
$\Sigma $ and for any invertible  $\OO_{\Sigma}$-sheaf $\LL'$
blowing ups of this scheme defined by sheaves $\I$ and $\I \otimes
\LL'$, are isomorphic. Then tensoring if necessary the sheaf of
ideals $\I$ by an appropriate invertible $\OO_T$-sheaf $\LL'$, we
achieve that sheaves in the upper row of the diagram (\ref{di})
coincide on the open subset out of subscheme of codimension $\ge
2$. This transformation leads to tensoring the sheaf $\widetilde
\L'$ by $\pi^{\ast} \LL'.$

After such a transformation the upper horizontal arrow in
(\ref{di}) is a canonical morphism of the sheaf
$p_{\ast}[(\E\otimes \L^m) \otimes
\sigma\!\!\!\sigma_{\ast}(\sigma\!\!\!\sigma^{-1}\I \cdot
\OO_{\widetilde \Sigma})^m]$ to its reflexive hull. The skew arrow
is the morphism of the same sheaf to a locally free sheaf which is
obviously reflexive. Hence this morphism factors through the
reflexive hull and gives rise to a (double dual to $\eta$)
morphism  $\eta^{\vee \vee}:p_{\ast}[\E\otimes \L^{m}] \to
\pi_{\ast}[\widetilde \E' \otimes (\widetilde \L')^m]$ of locally
free  $\OO_T$-sheaves. It is an isomorphism on open subset out of
a subscheme of codimension not less then 2. Besides, the
restrictions of both locally free sheaves to the reduction
$T_{red}$ also coincide. The coincidence of sheaves on the whole
of the scheme $T$ follows from the following simple algebraic
lemma.
\begin{lemma}\label{endmod} Let $A$ be a commutative ring,  $M$ finite $A$-module,
$\Phi:M \to M$ its $A$-endomorphism. Let the reduction
$\Phi_{red}: M_{red} \to M_{red}$ is $A_{red}$-automorphism. Then
$\Phi$ is $A$-automorphism.
\end{lemma}
\end{proof}
\begin{proof}[Proof of lemma \ref{endmod}] Since the homomorphism $\Phi_{red}$
is surjective and the nilradical includes in the Jackobson radical
then, due to  \cite[ch. 2, exercise 10]{AM}, $\Phi$ is also
surjective. Then by  \cite[theorem 2.4]{Mats} $\Phi$ is an
automorphism.
\end{proof}

\end{document}